\documentclass[reqno]{amsart}

\usepackage{amsmath,amssymb,amsthm}
\usepackage{setspace} 
\usepackage{enumerate}
\usepackage{graphicx}
\usepackage{dsfont}
\usepackage{fullpage}
\usepackage{floatrow}
\usepackage{array}
\usepackage{hyperref}
\usepackage{bookmark}
\usepackage{xcolor}

\hypersetup
{
colorlinks,
linkcolor={red!50!black},
citecolor={blue},
urlcolor={blue!80!black}
}

\allowdisplaybreaks
\newtheorem{theorem}{Theorem}[section]
\newtheorem{lemma}[theorem]{Lemma}
\theoremstyle{definition}
\newtheorem{definition}{Definition}[section]

\theoremstyle{proposition}
\newtheorem*{proposition*}{Proposition}

\def\E{{\mathbb E}}
\def\N{{\mathbb N}}
\def\P{{\mathbb P}}

\def\R{{\mathbb R}}

\def\J{{\mathfrak J}}

\def\O{{\mathcal O}}

\def\cK{{\mathrm K}}
\def\V{{\mathrm{Var}}}

\def\l{\ell}
\def\f{f}
\def\g{g}
\def\h{h}
\def\dis{\displaystyle}

\newcommand{\com}{,}
\newcommand{\ncomma}{, \ }
\newcommand{\ztp}{(0,2 \pi)}
\newcommand{\zp}{(0,\pi)}
\newcommand{\zpt}{(0,\pi/2)}
\newcommand{\Envtwo}{\E [ N_{n} \ztp ]}

\newcommand{\asntoinfty}{\quad\mbox{as } n\to\infty}

\newcommand{\zk}{\zeta_{k}}
\newcommand{\ak}{\alpha_{k}}
\newcommand{\bk}{\beta_{k}}
\newcommand{\Dk}{\Delta _{(k,n)} (t)}
\newcommand{\Dkak}{\Delta _{(k,n)} (\alpha_{k})}
\newcommand{\ul}{u_{\l}}
\newcommand{\etl}{\omega_\l}

\newcommand{\abs}[1]{\left\lvert #1 \right\rvert}

\newcommand{\ceil}[1]{\left\lceil #1 \right\rceil}

\begin{document}

\title[Random cosine polynomials with palindromic blocks]{Real zeros of random cosine polynomials with palindromic blocks of coefficients}

\author{Ali Pirhadi}
\address{Department of Mathematics, Oklahoma State University, Stillwater, OK 74078, USA}
%
\email{pirhadi@math.okstate.edu}







\begin{abstract}
It is well known that a random cosine polynomial $ V_n(x) = \sum_ {j=0} ^{n} a_j \cos (j x) \ncomma x \in \ztp $, with the coefficients being independent and identically distributed (i.i.d.) real-valued standard Gaussian random variables (asymptotically) has $ 2n / \sqrt{3} $ expected real roots. On the other hand, out of many ways to construct a dependent random polynomial, one is to force the coefficients to be palindromic. Hence, it makes sense to ask how many real zeros a random cosine polynomial (of degree $ n $) with identically and normally distributed coefficients possesses if the coefficients are sorted in palindromic blocks of a fixed length $ \l. $
In this paper, we show that the asymptotics of the expected number of real roots of such a polynomial is $ \cK_\l \cdot 2n / \sqrt{3} $, where the constant $ \cK_\l$ (depending only on $ \l $) is greater than 1, and can be explicitly represented by a double integral formula. That is to say, such polynomials have slightly more expected real zeros compared with the classical case with i.i.d. coefficients.
\end{abstract}

\maketitle



\section{Background}
Let $ I \subset \R $ be an interval, $\{f_j\}_{j=0}^n$ a sequence of real-valued functions on $ I $, and $\{a_j\}_{j=0}^n$ a sequence of random variables defined on the same probability space.
A real random function $ F_n: I \to \R $ is defined as the linear combination 
$
F_n (x) := \sum _{j=0} ^{n} a_j f_{j}(x) .
$
Among all famous random functions, of particular curiosity is the algebraic polynomial $ P_n $ which is simply determined as the linear combination of the monomials $ f_j(x)=x^j $, and is most-studied since it is easier to deal with than other types. The history of studying the real zeros of random (algebraic) polynomials dates back to the early 1930s with the pioneering works of Bloch and P\'olya \cite{BP}, Schur \cite{Sc} and Szeg\H{o} \cite{Sz}.
The first work in this area was actually done by Bloch and P\'olya \cite{BP} in 1932 where they proved that if the coefficients $a_j$ are i.i.d. random variables with $a_0=1$ almost surely (i.e. $\P(a_0=1)=1$) and all other coefficients are chosen uniformly from the set $\{-1,0,1\}$, then  $ \E[N_n(\R)]=\O \! \left(\sqrt{n}\right) $ as $ n \to \infty $, where $\E$ is the mathematical expectation, and $N_n(\R)$ denotes the number of real zeros of $P_n$ in $\R$. 
Later, in a series of papers (e.g. \cite{LO1} and \cite{LO2}) Littlewood and Offord provided lower and upper bounds for $N_n(\R)$ of a random polynomial with coefficients being uniformly distributed in the set $ \{ -1, 1 \} $ or in the interval $ ( -1, 1 ) $. Their results were later developed and refined by Erd\H{o}s and Offord \cite{EO} and others.
In 1943, Kac \cite{Kac1} showed that the asymptotics of the expected number of real zeros of polynomials with i.i.d. normal coefficients is $ (2/\pi + o(1) )\log n,$ which subsequently was studied and enhanced by other mathematicians, including Wang \cite{Wa}, Ibragimov and Maslova \cite{IM1}-\cite{IM2}, Edelman and Kostlan \cite{EK}, and Wilkins \cite{W1}.

An essential tool in study of the expected real zeros of a random polynomial with normal distribution is a celebrated formula known as Kac-Rice's Formula, first introduced by Kac \cite{Kac2} through an explicit integral formula. Despite the several forms of the formula, which are more or less equivalent, in this paper we quote and use the version indicated in the work of Lubinsky, Pritsker and Xie \cite{LPX}.
\begin{proposition*}[Kac-Rice Formula] 
Let $ [a, b] \subset \R ,$ and consider real valued functions $ f_{j}(x) \in C^1 [(a, b)] $,  $j = 0, \ldots , n $. Define the random function
$ 	F_n (x) := \sum _{j=0} ^{n} a_{j} f_{j}(x),  $
where the coefficients $a_j$ are i.i.d. random variables with Gaussian distribution $ \mathcal{N}(0,\sigma^{2}) $ and  $ A(x) := \sum _{j=0} ^{n} \left(f_{j} (x)\right)^2 $, $ B(x) := \sum _{j=0} ^{n} f_j (x) f'_{j} (x) $ and $ C(x) := \sum _{j=0} ^{n} \left(f'_{j} (x)\right)^2. $
If $ A(x)>0 $ on $ (a,b) ,$ and there is $ M \in \N $ such that $ F'_{n}(x) $
has at most $ M $ zeros in $ (a, b) $ for all choices of
coefficients, then the expected number of real zeros of $ F_n(x) $ in the interval $ (a, b) $ is given by
\begin{equation}  \label{Kac}
\E [ N_{ n }  ( a , b ) ] 
= \frac{1}{\pi} \displaystyle \int_{a}^{b}  \dfrac{ \sqrt{A(x)C(x)-B^2(x)}  }{A(x)} \, dx. \tag{\(\ast\)}
\end{equation}
\end{proposition*}

Expected real zeros of a random cosine polynomial 
\begin{equation*}
V_n(x) := \sum_{j=0}^{n} a_j \cos(j x) \ncomma x \in \ztp , 
\end{equation*}
with i.i.d. standard Gaussian coefficients was originally considered by Dunnage \cite{Dun} in 1966. 
He proved the following asymptotic result
\begin{equation*} 
\Envtwo =\frac{2n}{\sqrt{3}} + \O \! \left( n^{ 11/13} (\log n )^{3/13}\right) \quad\mbox{as } n \to\infty ,
\end{equation*}
which may be of course attained by setting $ f_{j} (x) = \cos (j x) $ and $ (a,b) =(0,2\pi) $ in the proposition above. Further results refining that estimate and sharpening its error term were obtained by Das \cite{D}, and the compelling work of Wilkins \cite{W2}, known as the best estimate up to this time, where he showed
\begin{equation*}
\E \left[ N_{n} ( 0,2 \pi ) \right] = \frac{2n+1}{ \sqrt{3}} 
\sum_{r=0}^{3} \dfrac{D_r}{(2n+1)^r}+ \O \! \left((2n+1)^{-3}\right) \asntoinfty
\end{equation*} 
with $ D_0 =1 $ and the other coefficients being explicitly computed. 
A moderately different type of a random trigonometric polynomial is called Qualls ensemble
\begin{equation*}
X_n(x) :=\frac{1}{\sqrt{n}}\sum _{j=1} ^{n}\left( a_{j} \cos (j x) + b_{j} \sin (j x)\right) \ncomma x \in \ztp .
\end{equation*}
Similar to the above results, Qualls \cite{Q} showed that the expected number of real zeros of the polynomial $ X_n $ is asymptotically $ 2n/\sqrt{3} $, and more importantly, he proved that
\begin{equation*}
\abs{N_{n} ( 0,2 \pi )-\E \left[ N_{n} ( 0,2 \pi ) \right]}\leqslant C n^{3/4}
\end{equation*}
for a large enough $ n $ and some positive constant $ C $--which is basically a huge improvement in the results obtained by Farahmand (cf. \cite{Far1-1}-\cite{Far1-2}).   
For the variance of the number of real roots of $ X_n $, the asymptotic form $ \V\left(N_{n} ( 0,2 \pi )\right)\sim cn $ (with $ c>0  $ being a constant) was conjectured by Bogomolny, Bohigas and Leboeuf \cite{BBL}, which was later proved by Granville and Wigman \cite{GW}. 
More on the history of the subject as well as many additional references and further directions of work on the variance and expected number of real zeros may be found in the books of Bharucha-Reid and Sambandham \cite{BS} and of Farahmand \cite{Far1-3}.
We now consider a family of identically and normally distributed coefficients $ (a_j)_{j=0}^n $ with zero mean and unit variance, and let $ \rho_{ij}=\E (a_i a_j ) $, $ i,j = 0, \cdots n $, be the correlation coefficients (also known as the moment matrix). It is trivial that $ [\rho_{ij}] $ is an $ (n+1)\times(n+1) $  identity matrix provided that $ (a_j)_{j=0}^n $ is a family of independent random variables.
Two cases of random cosine polynomials with dependent coefficients employing nontrivial correlation coefficients were established in the works of Sambandham \cite{Sam}, and Renganathan and Sambandham \cite{RS2}. The authors explained the cases of constant and geometric correlations 
$ \rho_{ij} = \rho $
and 
$ 
\rho_{ij} = \rho ^{  \abs{i - j} } ,
$ 
$ \rho \in ( 0, 1 ), $
$ i \ne j $, 
and showed in both cases the asymptotics of the mean number of real zeros remains as $  2n / \sqrt{3}. $
More recently Angst, Dalmao and Poly \cite{ADP} investigated the expected real zeros of a random trigonometric polynomial with a new correlation function $ \rho .$ They showed that, under mild assumptions of a spectral function $ \psi_\rho $, the asymptotic estimate $ \E \left[ N_{n} ( 0,2 \pi ) \right] = 2n/\sqrt{3}+o(n) $ holds as $ n $ grows to infinity. 
Compared with the exemplary case of a random cosine polynomial with i.i.d. Gaussian coefficients, to the best of our knowledge, the first result (in the field of random cosine polynomials with dependent coefficients) to possess more expected real roots appears in the work Farahmand and Li \cite{FL}.
The authors gave an asymptotic estimate for the expected number of real zeros of polynomial $ V_n $ having \emph{palindromic} properties 
$ 
a_{j} = a_{ n - j } $, and showed that $ \Envtwo 
\sim n / \sqrt{3} .$ 
Note that this result should be modified as 
\begin{equation} \label{a1}
\Envtwo 
= n + \dfrac{n}{\sqrt{3}} + \O \! \left( n^{3/4} \right) \asntoinfty 
\end{equation}
since we require taking $ n $ deterministic real zeros of $ V_n $ into account  (for further explanation, see comments on (1.1) of \cite{P}). The class of palindromic coefficients is just a special case of coefficients with pairwise equal blocks (defined in Section \ref{Sec2}) in view of the fact that any coefficient may be considered as a block of coefficients of unit length. Two random cosine polynomials with pairwise equal blocks of coefficients were studied in \cite{P}, and it turned out that in one case, the asymptotic expected real zeros remains the same as those of the classical case, whereas in the second, significantly more real zeros are expected.

In a few words, those results reveal that $ 2n / \sqrt{3} $ is the lower bound for the asymptotic expected real roots of a random cosine polynomial in the event that independence of the coefficients is removed. 
Besides investigating another case of dependent random cosine polynomials that endorses the above lower bound criterion, the main goal of this paper is to generalize and to formulate the asymptotics obtained by Farahmand and Li \cite{FL} if palindromic \emph{blocks} of coefficients are of interest in place of palindromic coefficients.
The plan of the article is the following: in Section \ref{Sec2}, we introduce  the class of coefficients with pairwise equal blocks in general, and the class of palindromic blocks of coefficients in particular, and proceed then with the main theorem, whose proof is given in Section \ref{Sec4}.

\section{Random cosine polynomials with palindromic blocks of coefficients} \label{Sec2}
Unlike a random cosine polynomial with palindromic coefficients, $ a_j=a_{n-j} $, whose asymptotic expected real zeros exceeds $ 2n / \sqrt{3} $, the asymptotics remains the same for a cosine polynomial with pairwise equal coefficients satisfying $ a_{2j}=a_{2j+1}$. 
Since any coefficient $ a_j $ may be considered as $ A_j= (a_j) $, or simply a block of coefficients of unit length, one can intuitively think of the expected real zeros of such polynomials while the coefficients are sorted in blocks of length $ \l $ instead of those of unit length. 
As anticipated, it turns out cosine polynomials with \emph{contiguous} blocks, $ A_{2j}=A_{2j+1}$, of length $ \l $ have the same asymptotics as those with i.i.d. coefficients (see Theorem 2.1 of \cite{P}). 
In this article we study the asymptotic expected real roots of a random cosine polynomial employing palindromic blocks of length $ \l $ in order to ascertain whether the number of real zeros increases, decreases or remains the same.

\begin{definition}
An $ n $-tuple $ \left(a_{i}, a_{i+1}, \ldots , a_{i + \l -1}\right) $ is called a block of coefficients of length $ \l $.
\end{definition}
Fix a positive integer $ \l $, and let $ A= \left(a_{0}, a_{1}, \ldots , a_{n}\right) \ne (0,0, \ldots,0) ,$ 
where $ n=2 \l m +r \ncomma m \in \N $, and $ r \in \{-1, \ldots , 2 \l -2 \}. $
We then establish $ \{ A_j \}_{j=0}^{2m-1} $, a sequence of $ 2m $ blocks of coefficients of length $ \l $, and the set of the remaining coefficients
$ \tilde{A} = A \setminus \bigcup_{j} A_j $
which contains $ r+1 $ elements while emphasizing that $ \tilde{A} = \emptyset $ if $ r=-1 $.
\begin{definition}
With above assumptions, suppose for each $ A_j  $ there exists a unique $ j' \neq j $ with $ A_{j} = A_{j'} ,$ and define
$ \J = \{ j=0 , \ldots , 2m-1 : A_{j} = A_{j'} \ \text{and} \ j < j' \}. $
If $ \bigcup_{j \in \J } A_j \cup \tilde{A} $ is a family of i.i.d. random variables, then $ A= \left(a_{0}, a_{1}, \ldots , a_{n}\right) $ is called a class of coefficients with pairwise equal blocks (of length $ \l $).
\end{definition}
\begin{definition}
Suppose $ A= \left(a_{0}, a_{1}, \ldots , a_{n}\right) $ is a class of coefficients with pairwise equal blocks of length $ \l $ with $ \J = \{ 0 , \ldots , m-1 \} $ and $ A_j= \left(a_{\l j}, a_{\l j+1}, \ldots , a_{\l j + \l -1}\right) \ncomma j \in \J. $
Then $ A= \left(a_{0}, a_{1}, \ldots , a_{n}\right) $ is a class of coefficients with palindromic blocks if the sequence of blocks $ A_0, \ldots, A_{m-1}, \tilde{A}, A_{m}, \ldots, A_{2m-1} $ is palindromic, i.e., $ A_j = A_{2m-1-j} \ncomma j \in \J $. Note that as $ r $ varies, we obtain different $ A_{m}, \ldots, A_{2m-1} $ and $ \tilde{A} $. 
\end{definition}
As a direct consequence of \eqref{a1}, it turns out that a random cosine polynomial with palindromic and normally distributed blocks of unit length, on average, has almost $ 36.6\% $ more real zeros than that of the classical case with independent coefficients.
The reason that the asymptotics in this case ($ \l=1 $) is significantly more than $ 2n/\sqrt{3} $ is that more than half of all real zeros are deterministic. As we expect, the longer a block is, the smaller number of real roots should be expected in contrast to \eqref{a1}.
For instance, a random cosine polynomial with palindromic blocks of lengths $ 2 $ and $ 3 $ respectively has only $ 6.42\% $ and $ 4.08\% $ more expected real roots than the classical case with i.i.d. coefficients. 

Herein we discuss the asymptotic expected real root of a random cosine $ V_n(x) = \sum _{j=0} ^{n} a_j \cos (j x) ,$ where $ A= \left(a_{0}, a_{1}, \ldots , a_{n}\right) $ is a class of coefficients with palindromic blocks of length $ \l $. The following theorem shows that in this case $ \Envtwo $ and $ 2n/\sqrt{3} $ are proportionally related by a constant depending only on $ \l $ and represented by a double integral formula. 

\subpdfbookmark{Theorem 2.1}{Thm2.1}
\begin{theorem} \label{Thm2.1}
Fix $ \l \in \mathbb{N} $, and let $ n= 2\l m +r \ncomma m \in \N ,$ 
and $ r \in \{-1, \ldots , 2 \l -2 \} .$
Assume $ V_n(x) = \sum _{j=0} ^{n} a_j \cos (j x) \ncomma x \in \ztp ,$
and $ \bigcup_{j=0}^{m-1} A_{j} \cup \tilde{A} $ is a family of i.i.d. random variables with Gaussian distribution $\mathcal{N}(0, \sigma^2) .$ For $ j=0 , \ldots , m-1 $ and $ k=0 , \ldots , \l-1 ,$ we further assume $  a_{\l j+k} = a_{\l (2m-1-j)+r+1+k} ,$ i.e., $ A_j = A_{2m-1-j}.$ Then
\begin{equation*}  
\Envtwo
=  \frac{2 n}{ \sqrt{3}} \ \cK_{\l} + \O  \! \left(n^{2/3}\right)  \asntoinfty \com
\end{equation*} 
where 
\begin{align*} 
& \cK_{\l} \notag
: = \frac{1}{\pi ^2} \displaystyle \int_{ 0 }^{ \pi } \int_{ 0 }^{ \pi /2 }
\sqrt{1+  \dfrac{3 \left(1-\ul^2(s)\right) }{  \left(1+  \ul(s) \cos( t )\right) ^2 }} \, ds \, dt ,
\end{align*} 
and 
$ \ul(s) := \dfrac{\sin (\l s)}{\l \sin(s)}. $
\end{theorem}
As a quick remark, it is worth mentioning that when $ \l =1 $ (and consequently $ \cK_{1}=1/2 $), we arrive at a random cosine polynomial with palindromic coefficients as in the work of Farahmand and Li (see \eqref{a1} and \cite{FL}). 
For a fixed $ \l \geqslant 2 $, Markov's inequality, which appears as Theorem 15.1.4 in \cite{RS}, gives that $ \abs{\ul(x)} \leqslant 1 $ on $ [0,\pi] $, and $ \abs{\ul(x)} = 1 $ only at the endpoints.
Hence, \cite[3.613(1) on p. 366]{GR} implies that
\begin{align*} 
& \dfrac{1}{\pi} \displaystyle \int_{ 0 }^{ \pi }
\dfrac{\sqrt{1-\ul^2(s)} \, dt}{1+  \ul(s) \cos( t )} 
= 1 \ncomma s\in(0,\pi).
\end{align*} 
Let us call the integrand in the above expression $ f_s $, and define $ \varphi(y) := \sqrt{1+3y^2} . $ We note that $ \varphi $ is a strictly convex function on $ [0,\infty) ,$ which implies that for any nonconstant $ f \in \mathbb{L}^1  ((0,\pi))$, Jensen's inequality is strict (see Theorem 3.3 of \cite{R}), i.e.,
\begin{align*} 
& \varphi \left(\displaystyle \dfrac{1}{\pi} \int_{ 0 }^{ \pi } f(t) \, dt \right) < \dfrac{1}{\pi} \int_{ 0 }^{ \pi } \varphi (f(t)) \, dt.
\end{align*} 
Therefore, applying the Fubini-Tunelli Theorem, we get
\begin{align*} 
1
& = \frac{1}{\pi ^2} \displaystyle \int_{ 0 }^{ \pi /2 }
\pi \, \varphi \left(\displaystyle \dfrac{1}{\pi} \int_{ 0 }^{ \pi } f_s(t) \, dt \right) ds
< 
\frac{1}{\pi ^2} \displaystyle \int_{ 0 }^{ \pi /2 }\int_{ 0 }^{ \pi } 
\varphi (f_s(t)) \, dt \, ds = \cK_{\l} \\
& \leqslant \frac{1}{\pi ^2} \displaystyle \int_{ 0 }^{ \pi /2 }\int_{ 0 }^{ \pi } 
\left(1+  \sqrt{3}\f_s(t)\right) dt \, ds 
= \frac{1+\sqrt{3}}{2}.
\end{align*} 
Shown in Table \ref{table1} are the numerical evaluations of $ \cK_{\l} $ for some values of $ \l. $
\begin{table}[h]
\renewcommand{\arraystretch}{1.3}
\begin{tabular}{ |c|c|c|c|c|c|c|c|c|c|}
\hline
$ \l $ & $ 2 $ & $ 3 $ & $ 5 $ & $ 7 $ & $ 8 $ & $ 15 $ & $ 30 $ & $ 100 $ & $ 2019 $\\
\hline
$ \cK_{\l} $ & $ 1.0642 $ & $ 1.0408 $ & $ 1.0239 $ & $ 1.0170 $ & $ 1.0148 $ & $ 1.0079 $ & $ 1.0039 $ & $ 1.0011 $ & $ 1.000046 $\\
\hline
\end{tabular}
\caption{Evaluation of $\cK_{\l}$}
\label{table1}
\end{table}	

\section{Proof of Theorem 2.1} \label{Sec4}
Before proving the main theorem, we need the following lemmas. The proofs of the first two lemmas can be found in \cite{P}.
\begin{lemma} \label{Lem3.1}
Let $ V_n(x) = \sum _{j=0} ^{n} a_j \cos (j x) $ with the $a_j$ being random variables with Gaussian distribution $\mathcal{N}(0, \sigma^2) $. If $ a \in (0, 1/2 ) $ is fixed, then 
$
\E \left[ N_{n} \! \left(0 , n^{-a}\right) \right] 
= \O  (n^{1-a})  
$
as $ n \to \infty. $	
\end{lemma}
The convenience of the above lemma is that we obtain a fairly strong conclusion for the expected real roots of a random cosine polynomial regardless of independence of its coefficients.

\begin{lemma}  \label{Lem3.2}
Fix $ a \in (0,1/2) $ and $ p \in \N ,$ 
and let 
$ 
x \in [ m^{-a}, \pi / p -m^{-a} ] \ncomma m \in \N.
$
For $ \lambda = 0 , 1 , 2 ,$ we define 
\begin{align*}
& P_{ \lambda } (p,m,x) 
:= \sum _{j=0} ^{m-1} j^{ \lambda }\cos(2p j)x ,
& &
& & Q_{ \lambda } (p,m,x) := \sum _{j=0} ^{m-1} j^ { \lambda } \sin(2p j)x , 
\\
& R_{ \lambda } (p,m,x) 
:= \sum _{j=0} ^{m-1} j^{ \lambda }\cos(2p j+1)x 
& & \text{and} 
& & S_{ \lambda } (p,m,x) 
:= \sum _{j=0} ^{m-1} j^{ \lambda }\sin(2p j+1)x .
\end{align*}
Then 
\begin{align*}
& P_{ \lambda } (p,m,x), \ Q_{ \lambda } (p,m,x) , \
R_{ \lambda } (p,m,x) \ \text{and} \ S_{ \lambda } (p,m,x)
= \O  \! \left(  m^ { \lambda+a} \right) \asntoinfty .
\end{align*}
\end{lemma}

The following lemma is crucial in proving Theorem \ref{Thm2.1}.
\subpdfbookmark{Lemma 3.3}{Lem3.3}
\begin{lemma} \label{Lem3.3}
Fix $ \l \in \N $ and $ c \in (0,1), $ and let 
$ F_{\l} = \dis \bigcup_{i=0}^{[\l/2]} \, [i \pi /\l - n^{-c}, i \pi /\l + n^{-c} ] $ and
$ \ul (x) =\dfrac{\sin(\l x)}{\l \sin(x)}.$
Then,
\begin{align*}
& \int_{F_\l}  \sqrt{1+  \dfrac{3 \left(1-\ul^2(x)\right) }{  \left(1+ \ul(x) \cos(nx) \right) ^2 }} \, dx =  \O  \! \left(n^{-c}\right) \asntoinfty.
\end{align*}		
Note that the implied constant (in Landau's notation) only depends on $ \l $.
\end{lemma}

\begin{proof}
The case $ \l =1 $ is trivial since $ u_1(x)=1 $. For a fixed $ \l \geqslant 2, $
let us call the integrand $ \g_{n},$ and define
\begin{equation*}
\f_{n} (x) :=   \frac{\sqrt{1-\ul^2(x)}}{1+ \ul(x) \cos(nx) } 
\, \ncomma x \in [0,\pi /2 ] .
\end{equation*}	
We note that
$
0 \leqslant  \g_{n} (x) \leqslant 1 + \sqrt{3} \, \f_{n} (x).
$
Thus, it is plausible to show that 
$
\int_{F_\l} \f_{n} (x) \, dx = \O  \! \left(n^{-c}\right) .
$
Markov's inequality (see Theorem 15.1.4 of \cite{RS}) gives that $ \abs{\ul(x)} \leqslant 1 $ on $ [0,\pi] $, and $ \abs{\ul(x)} = 1 $ only at $ x=0,\pi. $ Therefore, there exists $  0 < \etl <1 $ such that
\begin{equation} \label{d15}
\abs{\ul(x)} \leqslant \etl \ncomma x \in \left[ \pi/2\l,\pi- \pi/2\l\right] ,
\end{equation}	
which implies that $ 0 \leqslant \f_{n} (x) \leqslant 1/\left(1-\etl\right) $ on $ \left[ \pi/2\l,\pi/2\right] .$
Hence, it remains to show that
$
\int_{0}^{n^{-c}} \f_{n} (x) \, dx = \O  \! \left(n^{-c}\right) .
$
First, note that  
\begin{equation} \label{d16}
0 \leqslant \cos(\l x) \leqslant \ul(x) \leqslant \cos(x) \ncomma x \in \left[ 0, \pi/2\l\right] .
\end{equation}
Let $ n_1 = \ceil{n^{1-c} / \pi} ,$ and for $ 1 \leqslant k \leqslant n_1 $ define
\begin{align*}
& Q_{k}^{+} =\{ x \in [(k-1)\pi/n,k\pi/n] : \cos( n x) \geqslant 0 \}
& & \text{and}  
& Q_{k}^{-} =\{ x \in [(k-1)\pi/n,k\pi/n] : \cos( n x) \leqslant 0 \} .
\end{align*}
Since 
$ 0 \leqslant \f_{n} (x)  \leqslant 1 $ on $  Q_{k}^{+} , $ it is immediate that
$
\int_{Q_{k}^{+}} \f_{n} (x) \, dx \leqslant \pi /2n.
$
On the other hand, if $x \in Q_{k}^{-} $, then $ x \geqslant k\pi /2n $ and 
$ 
\sin(x/2)\geqslant x/\pi \geqslant k/ 2n 
$ 
which implies that
$ 
0 \leqslant \cos(x) \leqslant a_k := 1 - k^2/2n^2 <1 
$ 
on $ Q_{k}^{-} $. This combined with \eqref{d16} implies that
\begin{align*} 
\int_{Q_{k}^{-}}  \f_{n} (x) \, dx 
& \leqslant \int_{Q_{k}^{-}} \dfrac{ \sin(\l x) \, dx }{ 1+ \cos(x) \cos( n x) }
\leqslant \frac{k \l\pi}{n} \int_{Q_{k}^{-}} \dfrac{ dx }{ 1+ \cos(x) \cos( n x) } 
\leqslant \frac{k \l\pi}{n} \int_{Q_{k}^{-}} \dfrac{ dx }{ 1+ a_k \cos( n x)  } 
\\
& \leqslant \frac{k \l\pi}{n} \int_{(k-1) \pi / n}^{k \pi / n} \dfrac{dx}{ 1+ a_k \cos(nx)} 
= \frac{k \l\pi}{n^2} \int_{(k-1)\pi}^{k\pi} \dfrac{ dx }{ 1+ a_k \cos(x)  } .
\end{align*}
We use change of variables $ u = x-(k-1)\pi$ for the odd $ k ,$ similarly $ u = k\pi -x$ if $ k $ is even, and observe that
\begin{align} \label{d18}
& \int_{(k-1)\pi}^{k\pi} \dfrac{ dx }{ 1+ a_k \cos(x) } 
= \int_{0}^{\pi} \dfrac{ dx }{ 1+ a_k \cos(x) } = \frac{\pi}{\sqrt{1-a_k^2}},
\end{align}
where the latter equality is gained with help of \cite[3.613(1) on p. 366]{GR}. Hence, the last two relations imply that there exists $ N \in \N $ such that for all $ n \geqslant N $ and $ 1 \leqslant k \leqslant n_1 $,
\begin{align*}
\int_{Q_{k}^{-}}  \f_{n} (x) \, dx 
& \leqslant \frac{k \l \pi^2}{n^2 \sqrt{1-a_k^2}}
= \frac{k \l \pi^2}{n^2 \sqrt{\dfrac{k^2}{2n^2} \left( 2-\dfrac{k^2}{2n^2} \right)}}
= \frac{\sqrt{2} \l \pi^2}{n \sqrt{2-\dfrac{k^2}{2n^2}}}
\leqslant \frac{\sqrt{2} \l \pi^2}{n \sqrt{2-n^{-2c}}}
\leqslant \frac{\sqrt{2} \l \pi^2}{n} ,
\end{align*}
where the last two inequalities are obtained by the facts that $ n_1 \leqslant n^{1-c} $ and $n^{-2c} \to 0 $ respectively.
Therefore, there exists $ M_\l>0 $ such that, for all $ n \geqslant N ,$ we have
\begin{align} \label{d20}
\int_0^{n^{-c}} \f_{n} (x) \, dx
& \leqslant  \sum_{k=1}^{n_1} \int_{(k-1)\pi/n}^{k\pi/n} \f_{n} (x) \, dx
= \sum_{k=1}^{n_1} \int_{Q_{k}^{+} \cup Q_{k}^{-}} \f_{n} (x) \, dx  
\leqslant n^{1-c}  \left(\frac{\pi}{2n} + \frac{\sqrt{2} \l \pi^2}{n}\right) 
=  M_\l \, n^{-c} .
\end{align}
\end{proof}

\currentpdfbookmark{Proof of Theorem 2.1}{PF 2.1}
\begin{proof}[Proof of Theorem \ref{Thm2.1}] Fix $ a \in (0,1/2), $ and define $ \tilde{J} = \{ j : a_j \in \tilde{A}\} .$ 
For $ x \in \zp $, we see that
\begin{align*}
V_n(x)
& = \sum _{j=0} ^{n} a_j \cos (j x)
= \sum_{j=0}^{m-1} \sum_{k=0}^{\l -1} a_{\l j +k}  \left( \cos (\l j +k) x + \cos ( \l (2m-1-j) +r+1+k ) x \right) 
+ \sum_{j \in \tilde{J}} a_j \cos (j x) \\
& = \sum_{j=0}^{m-1} \sum_{k=0}^{\l -1} a_{\l j +k}  \left( \cos (\l j +k) x + \cos ( n -\l+1 -\l j +k ) x \right) 
+ \sum_{j \in \tilde{J}} a_j \cos (j x) ,
\end{align*} 
where 
$ 
\sum_{j \in \tilde{J}} a_j \cos (j x) = 0
$
if $ r=-1. $
For the sake of simplicity of our computations, we assume that $ n-\l $ is odd. 
Since either both $ \l j +k $ and $ n -\l+1 -\l j +k $ are odd or both are even, 
it is obvious that $ A(x) $, $ C(x) $ and $ B(x) $ are symmetric about $ x= \pi /2 ,$
which suggests  
$ \E [ N_{n} \zp ] = 2 \, \E [ N_{n} \zpt ]  $
due to the Kac-Rice Formula. 
It also follows from \cite[1.341(1,3) on p. 29]{GR} that, for all $ x \in \zp $,
\begin{align} \label{d9}
& \sum_{j=0}^{m-1}  \cos ( 2j + p ) x = \frac{\sin(mx) \cos( m-1+p)x }{\sin(x)}
& & \text{and}  
& \sum_{j=0}^{m-1}  \sin ( 2j + p ) x = \frac{\sin(mx) \sin( m-1+p)x }{\sin(x)} ,
\end{align}
especially
\begin{align} \label{d9-1}
& \sum_{j=0}^{m-1}  \cos ( 2\l j + p ) x = \frac{\sin(m\l x) \cos ((m-1)\l+p) x }{\sin(\l x)}.
\end{align}
Let $x \in E_{\l} = [0,\pi/2] \setminus F_{\l}$ with $ F_{\l} = \dis \bigcup_{i=0}^{[\l/2]}\, [i \pi /\l - n^{-a}, i \pi /\l + n^{-a} ] $. We observe that  
\begin{align*} 
A(x) 
& = \sum_{j=0}^{m-1} \sum_{k=0}^{\l -1} \left( \cos (\l j +k) x + \cos ( n -\l+1 -\l j +k ) x \right)^2 + \sum_{j \in \tilde{J}} \cos^2 (j x)\\
& = 4 \sum_{j=0}^{m-1} \cos ^2 \left(\frac{n -\l+1 -2\l j }{2}\right) \! x
\ \sum_{k=0}^{\l -1}  \cos ^2 \left(\frac{n -\l+1 +2k}{2}\right) \! x
+ \sum_{j \in \tilde{J}} \cos^2 (j x) \\
& = \sum_{j=0}^{m-1} \left( 1 + \cos(n -\l+1 -2\l j) x \right)  
\sum_{k=0}^{\l -1}  \left( 1+ \cos ( n- \l+1 +2k ) x \right)
+ \sum_{j \in \tilde{J}} \cos^2 (j x) \\
& = \sum_{j=0}^{m-1} \left( 1 + \cos(\l+r+1 +2\l j) x \right)  
\sum_{k=0}^{\l -1}  \left( 1+ \cos ( n- \l+1 +2k ) x \right)
+ \sum_{j \in \tilde{J}} \cos^2 (j x),
\end{align*}
where the last equality is obtained by replacing $ j $ with $ m-1-j $.
Next, applying \eqref{d9} and \eqref{d9-1} gives that
\begin{align*} 
A(x) 
& = \left( m + \frac{ \sin(m\l x) \cos(m\l+r+1)x }{\sin(\l x)} 
\right) \!
\left( \l + \frac{ \sin(\l x) \cos (nx)}{\sin(x)} \right) 
+ \sum_{j \in \tilde{J}} \cos^2 (j x)\\
& = m\l \left( 1 +  u_{m}(\l x) \cos(m\l+r+1)x 
\right)
\left( 1 +  \ul(x) \cos (nx)  \right)
+ \sum_{j \in \tilde{J}} \cos^2 (j x).
\end{align*}
Meanwhile, Markov's inequality shows $ \abs{u_{m}(\l x)} < 1 $ and $ \abs{\ul(x)} < 1 $ on $  E_{\l} $ implying that $ A(x)>0 $ on $ E_{\l} $. 
We note that
\begin{align*} 
A(x) 
& = \sum_{j=0}^{m-1} \sum_{k=0}^{\l -1} \left( \cos (\l j +k) x + \cos ( n -\l+1 -\l j +k ) x \right)^2 
+ \sum_{j \in \tilde{J}} \cos^2 (j x) \notag \\
& = \sum_{j=0}^{n} \cos^2 (j x)
+ 2 \sum_{j=0}^{m-1} \sum_{k=0}^{\l -1} \cos (\l j +k) x \cos ( n -\l+1 -\l j +k ) x \notag \\
& = \sum_{j=0}^{n} \cos^2 (j x)
+ \sum_{j=0}^{m-1} \sum_{k=0}^{\l -1} \cos ( n -\l+1 +2k ) x 
+ \sum_{j=0}^{m-1} \sum_{k=0}^{\l -1} \cos ( n -\l+1 -2\l j ) x \notag \\
& = \sum_{j=0}^{n} \cos^2 (j x)
+ \sum_{j=0}^{m-1} \sum_{k=0}^{\l -1} \cos ( n -\l+1 +2k ) x 
+ \sum_{j=0}^{m-1} \sum_{k=0}^{\l -1} \cos ( \l+r+1 +2\l j ) x ,
\end{align*}
where the last equality is obtained by replacing $ j $ with $ m-1-j $.
By Lemma \ref{Lem3.2}, it is clear that
\begin{align*} 
\sum_{j=0}^{n} \cos^2 (j x)
& = \frac{1}{2}\sum_{j=0}^{n} \left(1+ \cos (2j x)\right)
= \frac{n+1}{2} + \frac{P_0(1,n+1,x)}{2} 
= \frac{n}{2} +  \O  \! \left(n^{a}\right).
\end{align*}
Since $ \sec(x)= \O  \! \left(n^{a}\right) $ on $ E_\l $, \eqref{d9} implies that
\begin{align*} 
\sum_{j=0}^{m-1} \sum_{k=0}^{\l -1} \cos ( n-\l +1+2k ) x
& = m \sum_{k=0}^{\l -1} \cos ( n-\l +1+2k ) x 
= \frac{(n-r) \sin(\l x)\cos (nx)}{2\l \sin(x)} \notag \\
& = \frac{n \sin(\l x)\cos (nx)}{2\l \sin(x)} + \O  \! \left(n^{a}\right)
= \frac{n \ul(x) \cos(nx) }{2}  + \O  \! \left(n^{a}\right).
\end{align*}
The fact that $ \sec(\l x) = \O  \! \left(n^{a}\right) $ on $ E_\l $ along with \eqref{d9-1} helps us obtain
\begin{align*} 
&\sum_{j=0}^{m-1} \sum_{k=0}^{\l -1} \cos ( \l+r+1 +2\l j ) x 
= \frac{\l \sin(m\l x) \cos(m\l+r+1)x }{\sin(\l x)}
= \O  \! \left(n^{a}\right).
\end{align*}
Putting all the relations above together, we have
\begin{equation} \label{d21}
A(x) = \frac{n( 1 + \ul(x) \cos(nx) )}{2}  + \O  \! \left(n^{a}\right) \asntoinfty \ \text{and} \ x \in E_{\l}.
\end{equation}
In addition, we see 
\begin{align*} 
C(x) 
& = \sum_{j=0}^{m-1} \sum_{k=0}^{\l -1} \left( (\l j +k)\sin (\l j +k) x 
+ (  n -\l+1 -\l j +k ) \sin (  n -\l+1 -\l j +k ) x \right)^2 
+ \sum_{j \in \tilde{J}} j^2 \sin^2 (j x) \notag \\
& = \sum_{j=0}^{n} {j}^2 \sin ^2 (j x)
+ \sum_{j=0}^{m-1} \sum_{k=0}^{\l -1}  (\l j+k) ( n -\l+1 -\l j +k ) 
\left(\cos (n -\l+1 -2\l j ) x - \cos ( n -\l+1 + 2k ) x \right) .
\end{align*}
Note that 
\begin{align*} 
\sum_{j=0}^{n} {j}^2 \sin ^2 (j x)
& = \frac{1}{2}\sum_{j=0}^{n} {j}^2 - \frac{1}{2}\sum_{j=0}^{n} {j}^2 \cos(2j x)
= \frac{ n ( n+1 ) ( 2n+1) }{12} - \frac{P_{ 2 } (1,n+1,x)}{2} 
= \frac{ n ^3 }{6} + \O  \! \left(n^{2+a}\right) .
\end{align*}
Furthermore, the boundedness of $ \sum_{k=0}^{\l -1} k^{\lambda} \ncomma \lambda =1,2 ,$ together with Lemma \ref{Lem3.2} implies that
\begin{align*} 
& \sum_{j=0}^{m-1} \sum_{k=0}^{\l -1}  (\l j+k) ( n -\l+1 -\l j +k ) \cos (2j+1) \l x \notag \\
& = ( n -\l+1 ) \sum_{j=0}^{m-1} \sum_{k=0}^{\l -1}  (\l j+k) \cos (2j+1) \l x 
+ \sum_{j=0}^{m-1} \sum_{k=0}^{\l -1}  \left(k^2-\l^2j^2\right) \cos (2j+1) \l x
\notag \\
& = ( n -\l+1 ) \! \left(\l^2 R_{1} (1,m,\l x) + R_{0} (1,m,\l x) \sum_{k=0}^{\l -1} k \right)
+  R_{0} (1,m,\l x) \sum_{k=0}^{\l -1} k^2 - \l^3 R_{2} (1,m,\l x) 
= \O  \!  \left(n^{2+a}\right) .
\end{align*}
Similarly,
\begin{align*} 
& \sum_{j=0}^{m-1} \sum_{k=0}^{\l -1}  (\l j+k) ( n -\l+1 -\l j +k ) \sin (2j+1) \l x = \O  \!  \left(n^{2+a}\right) .
\end{align*}
The last two estimates combined with the identity $ \cos (n -\l+1 -2\l j ) x = \cos (n +1) x \cos (2j+1) \l x + \sin (n +1) x \sin (2j+1) \l x$ imply that
\begin{align*}  
\sum_{j=0}^{m-1} \sum_{k=0}^{\l -1}  (\l j+k) ( n -\l+1 -\l j +k ) 
\cos (n -\l+1 -2\l j ) x
= \O  \!  \left(n^{2+a}\right) .
\end{align*}
With help of \eqref{d9}, we also have
\begin{align*} 
& \sum_{j=0}^{m-1} \sum_{k=0}^{\l -1}  (\l j+k) ( n -\l+1 -\l j +k ) 
\cos (n -\l+1 +2k ) x \notag \\
& = \sum_{j=0}^{m-1} \sum_{k=0}^{\l -1}  \left(( n -\l+1 ) \l j-\l^2j^2\right) \cos (n -\l+1 +2k ) x 
+ \sum_{j=0}^{m-1} \sum_{k=0}^{\l -1}  \left((n-\l+1)k+k^2\right) \cos (n -\l+1 +2k ) x
\notag \\
& = \left(\frac{( n -\l+1 ) \l (m-1)m }{2} - \frac{ \l^2 (m-1)m(2m-1) }{6}\right)  \frac{\sin( \l x)  \cos( n x) }{\sin(x)} \notag \\
& + m \sum_{k=0}^{\l -1}  \left((n-\l+1)k+k^2\right) \cos (n -\l+1 +2k ) x.
\end{align*}
The fact that $ \sum_{k=0}^{\l -1} k^{\lambda} \cos ( n -\l+1 +2k ) x =\O(1) \ncomma \lambda =1,2 ,$ helps us rewrite above as
\begin{align*} 
& \sum_{j=0}^{m-1} \sum_{k=0}^{\l -1}  (\l j+k) ( n -\l+1 -\l j +k ) 
\cos (n -\l+1 +2k ) x \notag \\
& = \left(\frac{n \l^2 m^2 }{2} - \frac{ \l^3 m^3 }{3}\right)  \frac{\sin( \l x)  \cos( n x) }{\l \sin(x)}
+ \O  \! \left(m^{2}\right) = \frac{n^3 \ul(x) \cos(nx)}{12} + \O  \! \left(n^{2}\right).
\end{align*}
Hence, 
\begin{equation} \label{d27}
C(x) = \frac{ n ^3 ( 2-   \ul(x) \cos(nx) ) }{12} 
+ \O  \! \left(n^{2+a}\right) \asntoinfty \ \text{and} \ x \in E_{\l}.
\end{equation}
We also have
\begin{align*} 
B(x) 
& = - \sum_{j \in \tilde{J}} j \sin(j x) \cos(j x)
- \sum_{j=0}^{m-1} \sum_{k=0}^{\l -1} \left( \cos (\l j +k) x + \cos ( n -\l+1 -\l j +k ) x \right) \notag \\
& \cdot	\left( (\l j +k)\sin (\l j +k) x + ( n -\l+1 -\l j +k ) \sin ( n -\l+1 -\l j +k ) x \right) \notag \\
& = - \sum_{j=0}^{n} j \sin(j x) \cos(j x) 
- \sum_{j=0}^{m-1} \sum_{k=0}^{\l -1} (\l j+k) \sin (\l j+k) x \cos ( n -\l+1 -\l j +k ) x \notag \\
& - \sum_{j=0}^{m-1} \sum_{k=0}^{\l -1} (n -\l+1 -\l j +k ) \sin ( n -\l+1 -\l j +k ) x \cos (\l j+k) x .
\end{align*} 
Converting $ \sin(\, \cdot \,)\cos(\, \cdot \,) $ in the last two sums into addition, and after simplification, gives us
\begin{align*} 
B(x) 
& = - \sum_{j=0}^{n} j \sin(j x) \cos(j x) 
-\frac{1}{2} \sum_{j=0}^{m-1} \sum_{k=0}^{\l -1} (n -\l+1 + 2k) \sin ( n -\l+1 + 2k ) x \notag \\
& -\frac{1}{2} \sum_{j=0}^{m-1} \sum_{k=0}^{\l -1} (n -\l+1 -2\l j) \sin (n -\l+1 -2\l j) x. 
\end{align*} 
It is clear that
\begin{align*} 
& \sum_{j=0}^{n} j \sin(j x) \cos(j x) 
= \frac{1}{2} \sum_{j=0}^{n} j \sin(2j x) = \frac{Q_{1} (1,n+1,x)}{2} = \O  \! \left(n^{1+a}\right).
\end{align*} 
Since $ \sum_{k=0}^{\l -1} (-\l+1 + 2k) \sin (n -\l+1 + 2k) x =\O(1) $, we observe that
\begin{align*} 
& \sum_{j=0}^{m-1} \sum_{k=0}^{\l -1} (n -\l+1 + 2k) \sin ( n -\l+1 + 2k ) x 
=  n m \sum_{k=0}^{\l -1} \sin ( n -\l+1 + 2k ) x + \O(m) \notag \\	
& = \frac{ n(n-r) \sin( \l x) \sin( n x) }{2 \l \sin(x)} + \O(m)
= \frac{n^2 \ul(x) \sin( n x) }{2} + \O(n) \com
\end{align*} 
with the second-to-last equality being obtained by \eqref{d9}.
We also note that
\begin{align*} 
\sum_{j=0}^{m-1} \sum_{k=0}^{\l -1} (n -\l+1 -2\l j) \sin (2j+1) \l  x
& = \l (n -\l+1) \sum_{j=0}^{m-1} \sin (2j+1) \l  x 
- 2\l^2 \sum_{j=0}^{m-1} j \sin (2j+1) \l  x\notag \\
& = \l (n -\l+1)  S_{0} (1,m,\l x) - 2\l^2 S_{1} (1,m,\l x)
= \O  \! \left(n^{1+a}\right) .
\end{align*} 
In much the same way,
\begin{align*} 
& \sum_{j=0}^{m-1} \sum_{k=0}^{\l -1} (n -\l+1 -2\l j) \cos (2j+1) \l  x
= \O  \! \left(n^{1+a}\right) .
\end{align*}
Thus, the last two estimates imply that
\begin{align*} 
& \sum_{j=0}^{m-1} \sum_{k=0}^{\l -1} (n -\l+1 -2\l j) \sin (n -\l+1 -2\l j) x
= \O  \! \left(n^{1+a}\right) .
\end{align*}
Therefore,
\begin{equation} \label{d33}
B(x) 
= - \frac{n^2 \ul(x) \sin( n x)}{4} + \O  \! \left(n^{1+a}\right) \asntoinfty \ \text{and} \ x \in E_{\l}.
\end{equation}
Hence, \eqref{d21}-\eqref{d33} imply that
\begin{align*} 
\sqrt{ A(x) C(x) - B^2(x) } 
& = \frac{n^2}{4 \sqrt{3} }  \sqrt{\left( 1+ \ul(x) \cos( n x) \right)^2 +3 \left( 1 - \ul^2(x) \right)}  \\
& + \O  \! \left(n^{1+a}\right) \asntoinfty \ \text{and} \ x \in E_{\l}.
\end{align*}
Recall that $ \E [ N_{n} \zp ] = 2 \, \E [ N_{n} \zpt ]  $
and 
$
\E [ N_n ( F_{\l} ) ] = \O  \! \left(n^{1-a}\right)  
$ 
by Lemma \ref{Lem3.1}. Therefore, the Kac-Rice Formula \eqref{Kac} combined with the fact that 
$ A(x)>0 $ on $ E_{\l} $
implies that
\begin{align*} 
\E [ N_{n} \ztp ] & =
4 \, \E [ N_{n} \zpt ]  
= \frac{4}{\pi} \displaystyle \int_{E_{\l}} \dfrac{ \sqrt{A(x)C(x)-B^2(x)} }{A(x)} \, dx
+ \O  \! \left(n^{1-a}\right) \notag \\ 
& = \frac{2n}{ \sqrt{3} } \left( \frac{1}{\pi} \int_{E_{\l}}
\sqrt{1+  \dfrac{3 \left(1-\ul^2(x)\right) }{  \left(1+ \ul(x) \cos(nx) \right) ^2 }} \, dx \right)  + \O  \! \left(n^{1-a}\right) \asntoinfty , 
\end{align*}
and Lemma \ref{Lem3.3} helps us write
\begin{align} \label{d35}
& \E [ N_{n} \ztp ] 
= \frac{2n}{ \sqrt{3} } \left( \frac{1}{\pi} \int_{0}^{\pi/2}
\sqrt{1+  \dfrac{3 \left(1-\ul^2(x)\right) }{  \left(1+ \ul(x) \cos(nx) \right) ^2 }} \, dx \right) + \O  \! \left(n^{1-a}\right) \quad\mbox{as } n\to\infty .
\end{align}
Note that allowing $ \l=1 $ gives us
$ \E [ N_{n} \ztp ] \sim n/\sqrt{3}, $ as in the work of Farahmand and Li \cite{FL}. 
For a fixed $ \l \geqslant 2 $ and $ n \in \N $, we define
\begin{equation*} 
I_{\l}(n) 
:= \frac{1}{\pi} \int_{0}^{\pi /2}
\sqrt{1+  \dfrac{3 \left(1-\ul^2(x)\right) }{  \left(1+ \ul(x) \cos(nx) \right) ^2 }} \, dx .
\end{equation*}
Our aim is to show that 
\begin{equation} \label{d36}
I_{\l}(n) = \cK_{\l} + \O  \! \left(n^{-1/3}\right) \asntoinfty .
\end{equation} 
Assume $ N\in \N $ as in \eqref{d20}, and let $ n'=[n/2] $, $ n \geqslant N $, and divide $ [0, \pi] $ into $ n $ subintervals. Thus,
\begin{align*}  
I_{\l}(n) 
& = \frac{1}{\pi} \sum_{k=1}^{n'} \int_{(k-1) \pi /n }^{k \pi /n } \sqrt{1+  \dfrac{3 \left(1-\ul^2(x)\right) }{  \left(1+ \ul(x) \cos(nx) \right) ^2 }} \, dx 
+ \frac{1}{\pi} \int_{ n' \pi /n }^{ \pi / 2 } 
\sqrt{1+  \dfrac{3 \left(1-\ul^2(x)\right) }{  \left(1+ \ul(x) \cos(nx) \right) ^2 }} \, dx .
\end{align*}
We note that if $ n $ is even, the latter integral vanishes, and it could be made as small as $ \O  \! \left(n^{-1}\right)$ if $ n $ is odd, cf. Lemma \ref{Lem3.3}. So, for simplicity we let $ n $ be even.
The main tool in reaching our goal is to use the Riemann sum. Let $ t=nx $, and observe that
\begin{align*}  
I_{\l}(n)
& =  \frac{1}{n \pi} \sum_{k=1}^{n/2} \displaystyle \int_{ (k-1) \pi }^{k \pi  }  \sqrt{1+  \dfrac{3 \left(1-\ul^2(t/n)\right) }{  \left(1+ \ul(t/n) \cos(t) \right) ^2 }} \, dt .
\end{align*}
In the interval $ [(k-1) \pi,k \pi ] \ncomma k=1, \ldots, n/2 $, we let
$ \zk = k \pi $ if $k$ is odd, and $ \zk = (k-1) \pi $ if $k$ is even.
Note that the $ \zk $ are to be used in the Riemann sum approximating $ I_\l(n). $ First, we would like to show that
\begin{equation*} 
I_{\l}(n)
=  \frac{1}{n \pi} \sum_{k=1}^{n/2} \displaystyle \int_{ (k-1) \pi }^{k \pi  }  \sqrt{1+  \dfrac{3 \left(1-\ul^2(\zk/n)\right) }{  \left(1+ \ul(\zk/n) \cos(t) \right) ^2 }} \, dt +  \O  \! \left(n^{-1/3}\right) \asntoinfty \com
\end{equation*}
or equivalently, we wish to prove that
\begin{align}  \label{d40}
& \frac{1}{n \pi} \sum_{k=1}^{n/2} \displaystyle \int_{ (k-1) \pi }^{k \pi  }  \Dk \, dt 
= \O  \! \left(n^{-1/3}\right) \asntoinfty \com
\end{align}
where
\begin{equation*} 
\Delta _{(k,n)} (t)
:=  \sqrt{1+  \dfrac{3 \left(1-\ul^2(t/n)\right) }{  \left(1+ \ul(t/n) \cos(t) \right) ^2 }} 
- \sqrt{1+  \dfrac{3 \left(1-\ul^2(\zk/n)\right) }{  \left(1+ \ul(\zk/n) \cos(t) \right) ^2 }} \com \ t \in [(k-1) \pi,k \pi].
\end{equation*}
Let $ n_1= \ceil{n^{1-b}/\pi} $ and $ n_2= [n/2\l] $, $ n \geqslant N $, and for our fixed $ a \in (0,1/2), $ let $ b =(1-a)/2 $.
If $  1 \leqslant k \leqslant n_1 $, replacing $ a_k $ with $ \ul(\zk/n) $ in \eqref{d18} suggests that
\begin{align*} 
\int_{ (k-1) \pi }^{k \pi  }  \sqrt{1+  \dfrac{3 \left(1-\ul^2(\zk/n)\right) }{  \left(1+ \ul(\zk/n) \cos(t) \right) ^2 }} \, dt 
& \leqslant \pi + \sqrt{3} \int_{ (k-1) \pi }^{k \pi} \frac{\sqrt{1-\ul^2(\zk/n)} \, dt}{1+ \ul(\zk/n) \cos(t) }  \notag \\
& =  \pi + \sqrt{3} \int_{0}^{\pi} \frac{\sqrt{1-\ul^2(\zk/n)} \, dt}{1+ \ul(\zk/n) \cos(t) } = (1 + \sqrt{3}) \pi .
\end{align*}
It follows from $ n_1 \leqslant n^{1-b} $ and the inequality above that
\begin{align*}  
\frac{1}{n \pi} \sum_{k=1}^{n_1} \displaystyle \int_{ (k-1) \pi }^{k \pi  }  \sqrt{1+  \dfrac{3 \left(1-\ul^2(\zk/n)\right) }{  \left(1+ \ul(\zk/n) \cos(t) \right) ^2 }} \, dt 
& \leqslant (1 + \sqrt{3}) n^{-b} .
\end{align*}
Replacing $ c $ with $ b $ in \eqref{d20} also gives us
\begin{align*}  
& \frac{1}{n} \sum_{k=1}^{n_1} \displaystyle \int_{ (k-1) \pi }^{k \pi  }  \sqrt{1+  \dfrac{3 \left(1-\ul^2(t/n)\right) }{  \left(1+ \ul(t/n) \cos(t) \right) ^2 }} \, dt 
=  \sum_{k=1}^{n_1} \int_{(k-1)\pi/n}^{k\pi/n} \sqrt{1+  \dfrac{3 \left(1-\ul^2(x)\right) }{  \left(1+ \ul(x) \cos(nx) \right) ^2 }} \, dx \notag \\
& = \sum_{k=1}^{n_1} \int_{(k-1)\pi/n}^{k\pi/n} \g_{n}(x) \, dx
\leqslant \pi n^{-b} 
+ \sqrt{3} \, \sum_{k=1}^{n_1} \int_{(k-1)\pi/n}^{k\pi/n} \f_{n}(x) \, dx
\leqslant \left(\pi +\sqrt{3} M_\l\right) n^{-b} .
\end{align*}
Thus, the last two relations together with the triangle inequality ensure that 
\begin{align}  \label{d44}
& \frac{1}{n \pi} \sum_{k=1}^{n_1} \displaystyle \int_{ (k-1) \pi }^{k \pi  }  \Dk \, dt = \O  \! \left(n^{-b}\right) \asntoinfty. 
\end{align}
On the other hand, if $ n_1+1 \leqslant k \leqslant n/2 $, it is clear that 
$
\Delta _{(k,n)}  \in C^{1} \! \left([(k-1) \pi,k \pi]\right).
$
Thus, by the Mean Value Theorem for Integrals there exists	$\ak \in \left((k-1) \pi,k \pi\right)$ such that 
\begin{align*}  
& \abs{ \frac{1}{n \pi} \sum_{k=n_1+1}^{n/2} \displaystyle \int_{ (k-1) \pi }^{k \pi  } \Delta _{(k,n)}(t) \, dt } 
= \abs{ \frac{1}{n} \sum_{n_1+1}^{n/2} \Dkak } 
\leqslant \frac{1}{n } \sum_{n_1+1}^{n/2} \abs{\Dkak}. 
\end{align*} 
For $ n_1+1 \leqslant k \leqslant n/2 $, we define
\begin{equation*} 
\f_{\left(k,n\right)}(t) := \frac{\sqrt{1-\ul^2(t/n)}}{1+ \ul(t/n) \cos(\ak) } \,  \ncomma t \in [(k-1) \pi,k \pi].
\end{equation*}
It is easy to show that 
\begin{align*}  
\abs{\Dkak}
& = \abs{ \sqrt{1+3 \f_{(k,n)}^2 (\ak)} -\sqrt{1+3  \f_{(k,n)}^2 (\zk)} \, } 
= 3 \abs{ \f_{(k,n)} (\ak) - \f_{(k,n)} (\zk)}
\notag \\
& \cdot \dfrac{\f_{(k,n)} (\ak) + \f_{(k,n)} (\zk)}
{\sqrt{1+3 \f_{(k,n)}^2 (\ak)} + \sqrt{1+3 \f_{(k,n)}^2 (\zk)}} 
\leqslant \sqrt{3} \abs{ \f_{(k,n)} (\ak) - \f_{(k,n)} (\zk)}. 
\end{align*}
Since $ 0 \leqslant \abs{\ul(t/n)}  <1$ on $ [(k-1) \pi,k \pi] $, $ n_1+1 \leqslant k \leqslant n/2 $, it is clear that $ 0 <\f_{\left(k,n\right)} \in C^{1} \! \left([(k-1) \pi,k \pi]\right) . $
Hence, by the Mean Value Theorem there exists 
$ \bk $ lying between $ \ak $ and $ \zk $ such that
\begin{align*}  
\abs{ \f_{(k,n)} (\ak) - \f_{(k,n)} (\zk)}
& =  \abs{\zk -\ak} \cdot \abs{\f'_{\left(k,n\right)}(\bk)} 
\leqslant \pi \abs{\f'_{\left(k,n\right)}(\bk)}  .
\end{align*}
Therefore,
\begin{align} \label{d47}
\abs{ \frac{1}{n \pi} \sum_{k=n_1+1}^{n/2} \displaystyle \int_{ (k-1) \pi }^{k \pi  } \Delta _{(k,n)}(t) \, dt } 
& \leqslant \frac{ \sqrt{3} \pi}{n} \sum_{n_1+1}^{n/2} \abs{\f'_{\left(k,n\right)}(\bk)}.
\end{align}
We also note that
\begin{align} \label{d47-2}
\f'_{\left(k,n\right)}(\bk) 
& = \dfrac{ \ul^2(\bk/n) \cos(\bk/n)-\ul(\bk/n)\cos(\l \bk/n) }{n \sin(\bk/n) \sqrt{1-\ul^2(\bk/n)} \left(1+ \ul(\bk/n) \cos(\ak)\right) } \notag \\
& + \dfrac{ \sqrt{1-\ul^2(\bk/n)} \left(\ul(\bk/n)\cos(\bk/n)\cos(\ak) - \cos(\l \bk/n)\cos(\ak)\right)}{n \sin(\bk/n) \left(1+ \ul(\bk/n) \cos(\ak)\right)^2 } \notag \\
& = \frac{1}{n} \left( \dfrac{ \ul(\bk/n) \cos(\bk/n)-\cos(\l \bk/n) }{\sin(\bk/n) \sqrt{1-\ul^2(\bk/n)}} \right) \!
\left( \dfrac{ \ul(\bk/n) }{1+ \ul(\bk/n) \cos(\ak)} + \dfrac{ \left(1-\ul^2(\bk/n) \right) \cos(\ak) }{\left(1+ \ul(\bk/n) \cos(\ak)\right)^2} \right) \notag \\
& =: \frac{1}{n} \, \h_\l(\bk/n)
\left( \dfrac{ \ul(\bk/n) }{1+ \ul(\bk/n) \cos(\ak)} + \dfrac{ \left(1-\ul^2(\bk/n) \right) \cos(\ak) }{\left(1+ \ul(\bk/n) \cos(\ak)\right)^2} \right).
\end{align}
First, we assume that $ n_1+1 \leqslant k \leqslant n_2, $ which implies that  $n^{-b} \leqslant \bk/n \leqslant \pi/2\l . $ 
We note that
\begin{align*} 
\h_\l(x)
& = \dfrac{ \ul(x) \cos(x)-\cos(\l x) }{\sin(x) \sqrt{1-\ul^2(x)}}
= \dfrac{ \sin(\l x)\cos(x) -\l \sin(x)\cos(\l x) }{\l\sin^2(x) \sqrt{1-\ul^2(x)}} \\
& =\dfrac{ (1-\l)\sin(\l +1)x + (1+\l)\sin(\l -1)x }{2\l\sin^2(x) \sqrt{1-\ul^2(x)}}.
\end{align*}
Since $ (1-\l)\sin(\l +1)x + (1+\l)\sin(\l -1)x$ vanishes at $ x=0 $, and
\begin{align*} 
\dfrac{d}{dx}\left((1-\l)\sin(\l +1)x + (1+\l)\sin(\l -1)x\right)= 2(\l^2-1)\sin(\l x)\sin(x) \geqslant 0 \ncomma x \in \left[0,\pi/2\l\right],
\end{align*}
we see that $ \h_\l(x) $ is nonnegative on $ \left[0,\pi/2\l\right] $. Hence, We apply \eqref{d16}, and see that
\begin{align} \label{d47-3}
& 0 \leqslant \h_\l(x)
\leqslant \dfrac{ \cos^2(x)-\cos(\l x) }{\sin^2(x)} \leqslant \frac{\l^2-2}{2} \ncomma x \in \left[0,\pi/2\l\right],
\end{align}
where the last inequality follows from the facts that $ \dfrac{ \cos^2(x)-\cos(\l x) }{\sin^2(x)}  $ is a strictly decreasing function on $ \left[0,\pi/2\l\right] $ and $\dis \lim_{x\to 0^+}\dfrac{ \cos^2(x)-\cos(\l x) }{\sin^2(x)} = \frac{\l^2-2}{2}. $ It follows from \eqref{d47-2}, \eqref{d47-3} and \eqref{d16} that, for $ n_1+1 \leqslant k \leqslant n_2, $
\begin{align*} 
\abs{\f'_{\left(k,n\right)}(\bk) }
& \leqslant \frac{\l^2-2}{2n} 
\left( \dfrac{1}{ 1-\ul(\bk/n) } 
+ \dfrac{ 1-\ul^2(\bk/n)}{ \left(1-\ul(\bk/n) \right)^2 } \right) 
= \dfrac{ \left(\l^2-2\right) \left(2+\ul(\bk/n)\right)}{2n\left(1-\ul(\bk/n)\right)}  \notag \\
& \leqslant \frac{3\left(\l^2-2\right)}{2n\left(1 - \cos(\bk/n)\right)} 
= \frac{3\left(\l^2-2\right)}{4n\sin^2(\bk/2n)} 
\leqslant  \frac{3\left(\l^2-2\right)\pi^2}{4n \left(\bk/n\right)^2}
\leqslant  \frac{3\left(\l^2-2\right)\pi^2n^{-1+2b} }{4} \notag \\
& = \frac{3\left(\l^2-2\right)\pi^2n^{-a} }{4},
\end{align*}
and with help of \eqref{d47}, we have
\begin{align*}
\abs{ \frac{1}{n \pi} \sum_{k=n_1+1}^{n_2} \displaystyle \int_{ (k-1) \pi }^{k \pi  } \Delta _{(k,n)}(t) \, dt } 
& \leqslant \frac{ \sqrt{3} \pi}{n} \sum_{n_1+1}^{n_2} \abs{\f'_{\left(k,n\right)}(\bk)}
\leqslant \frac{3\sqrt{3}\left(\l^2-2\right)\pi^3 n^{-a} }{8\l} =C_\l \, n^{-a} .
\end{align*}
Thus,
\begin{align}  \label{d48}
& \frac{1}{n \pi} \sum_{k=n_1+1}^{n_2} \displaystyle \int_{ (k-1) \pi }^{k \pi  }  \Dk \, dt = \O  \! \left(n^{-a}\right) \asntoinfty. 
\end{align}
Next, we assume that $ n_2+1 \leqslant k \leqslant n/2, $ which implies that  $\pi/2\l \leqslant \bk/n \leqslant \pi/2 . $
Thus, \eqref{d15} and \eqref{d47-2} give that, for $ n_2+1 \leqslant k \leqslant n/2, $
\begin{align*} 
\abs{\f'_{\left(k,n\right)}(\bk) }
& \leqslant \dfrac{ \h_\l(\bk/n) \left(2+\ul(\bk/n)\right)}{n\left(1-\ul(\bk/n)\right)} 
\leqslant \frac{2\left(2+\etl\right)}{n \sin(\bk/n) \sqrt{1-\etl^2}\left(1-\etl \right)} 
\leqslant \frac{6\l}{n \left(1-\etl\right)^{3/2}} ,
\end{align*}
with the last inequality following from $ \sin(\bk/n) \geqslant 2\bk /n\pi \geqslant 1/\l $ for $ n_2+1 \leqslant k \leqslant n/2. $
Thus, with help of \eqref{d47}, we have
\begin{align*}
\abs{ \frac{1}{n \pi} \sum_{k=n_2+1}^{n/2} \displaystyle \int_{ (k-1) \pi }^{k \pi  } \Delta _{(k,n)}(t) \, dt } 
& \leqslant \frac{3\sqrt{3}\l\pi}{n \left(1-\etl\right)^{3/2}} =D_\l \, n^{-1} ,
\end{align*}
which implies that
\begin{align}  \label{d48-2}
& \frac{1}{n \pi} \sum_{k=n_2+1}^{n/2} \displaystyle \int_{ (k-1) \pi }^{k \pi  }  \Dk \, dt = \O  \! \left(n^{-1}\right) \asntoinfty. 
\end{align}
Therefore, \eqref{d44}, \eqref{d48} and \eqref{d48-2} give that
\begin{align*}  
& \frac{1}{n \pi} \sum_{k=1}^{n/2} \displaystyle \int_{ (k-1) \pi }^{k \pi  }  \Dk \, dt = \O  \! \left(n^{-b}\right) + \O  \! \left(n^{-a}\right) 
 + \O  \! \left(n^{-1}\right) = \O  \! \left(n^{-b}\right) + \O  \! \left(n^{-a}\right) \asntoinfty. 
\end{align*}
Since $ b=(1-a)/2 \ncomma a \in (0,1/2) $, it is clear that the best estimate occurs when $ a=b=1/3. $ So 
\begin{align*}  
& \frac{1}{n \pi} \sum_{k=1}^{n/2} \displaystyle \int_{ (k-1) \pi }^{k \pi  }  \Dk \, dt = \O  \! \left(n^{-1/3}\right) \asntoinfty. 
\end{align*}
So far, we proved that \eqref{d40} holds. In other words,
\begin{equation*} 
I_{\l}(n)
=  \frac{1}{n \pi} \sum_{k=1}^{n/2} \displaystyle \int_{ (k-1) \pi }^{k \pi  }  \sqrt{1+  \dfrac{3 \left(1-\ul^2(\zk/n)\right) }{  \left(1+ \ul(\zk/n) \cos(t) \right) ^2 }} \, dt +  \O  \! \left(n^{-1/3}\right) \asntoinfty .
\end{equation*}
We use the same change of variables as in \eqref{d18} and observe that
\begin{align} \label{e14}
I_{\l}(n)
& =  \frac{1}{n \pi} \sum_{k=1}^{n/2} \displaystyle \int_{0}^{\pi}  \sqrt{1+  \dfrac{3 \left(1-\ul^2(\zk/n)\right) }{  \left(1+ \ul(\zk/n) \cos(t) \right) ^2 }} \, dt +  \O  \! \left(n^{-1/3}\right) \notag \\
& =  \frac{1}{\pi^2}  \int_{ 0 }^{ \pi  } \frac{\pi}{n} \sum_{k=1}^{n/2} \sqrt{1+  \dfrac{3 \left(1-\ul^2(\zk/n)\right) }{  \left(1+ \ul(\zk/n) \cos(t) \right) ^2 }} \,  dt +  \O  \! \left(n^{-1/3}\right)
\asntoinfty.
\end{align}
We next show that, for sufficiently large $ n $, 
\begin{align} \label{e1}
\int_{ 0 }^{ \pi  } \frac{\pi}{n} \sum_{k=1}^{n} \sqrt{1+  \dfrac{3 \left(1-\ul^2(\zk/n)\right) }{  \left(1+ \ul(\zk/n) \cos(t) \right) ^2 }} \, dt  
& = \int_{ 0 }^{ \pi  } \int_{ 0 }^{ \pi } \sqrt{1+  \dfrac{3 \left(1-\ul^2(s)\right) }{  \left(1+ \ul(s) \cos(t) \right) ^2 }} \, ds 
\, dt 
+  \O  \! \left( \frac{\log n}{n} \right) .
\end{align}
We define 
\begin{equation*}
\g(s,t) := \sqrt{1+  \dfrac{3 \left(1-\ul^2(s)\right) }{  \left(1+ \ul(s) \cos(t) \right) ^2 }} \, \ncomma (s,t) \in (0,\pi) ^2.
\end{equation*}
For any fixed $ t \in (0,\pi) $, let $ \g_{t}(s) := \g(s,t) \ncomma s \in (0,\pi) ,$ and $ I(t) := \int_{ 0 }^{ \pi  } \g_{t}(s) \, ds .$ For $ n'=n/2 $, we also define
\begin{align*} 
R_{n'}(\g_{t})
& :=  \frac{\pi}{n'} \sum_{k=1}^{n'} \g_{t} \left( (k-1/2) \pi/n' \right)
=  \frac{2 \pi}{n} \sum_{k=1}^{n'} \g_{t} ( (2k-1) \pi/n ) =  \frac{\pi}{n} \sum_{k=1}^{n'} 2 \g_{t} \left( (2k-1) \pi/n \right)
=  \frac{\pi}{n} \sum_{k=1}^{n} \g_{t} ( \zk/n ).
\end{align*}	
Therefore, proving \eqref{e1} requires showing
\begin{align} \label{e12}
\int_{ 0 }^{ \pi  } R_{n'}(\g_{t}) \, dt  
& = \int_{ 0 }^{ \pi  } I(t) \, dt  +  \O  \! \left( \frac{\log n}{n} \right) \asntoinfty.
\end{align}
We define
\begin{equation*}
\f(s,t) := \dfrac{ \sqrt{1-\ul^2(s)} }{ 1+ \ul(s) \cos(t)  } \, \ncomma (s,t) \in (0,\pi) ^2 \com
\end{equation*}
and for any fixed $ t \in (0,\pi) $ let $ \f_{t}(s) := \f(s,t) \ncomma s \in (0,\pi) .$ 
It is clear that 
$
0 < \sqrt{3} \f_{t}(s) \leqslant \g_{t}(s) \leqslant 1 + \sqrt{3} \f_{t}(s)
$ 
and
\begin{align*} 
\abs{\frac{d \g_t(s)}{d s}} 
& = \frac{3 \f_{t}(s)}{\g_{t}(s)}\abs{\frac{d \f_{t}(s)}{d s}} 
\leqslant \sqrt{3} \, \abs{\frac{d \f_{t}(s)}{d s}} .
\end{align*}
Thus, the sufficient condition for $ \g_{t}$, $ t\in(0,\pi) $, being integrable and a function of bounded variation over $ [0,\pi] $ is $ \f_{t} $ being integrable and of bounded variation over $ [0,\pi] $.
Note that any computation below which is valid for $ t \in (0, \pi/2) $ subsequently holds for  $ t \in (\pi/2,\pi) $.
It is trivial that $ \f_{\pi /2}(s) = \sqrt{1-\ul^2(s)} \leqslant 1 $ is integrable over $ [0,\pi] $. 
Fix $ t \in (0,\pi/2) $. It is clear that $ \f_{t}(s) \leqslant 1$ for all $ s\in [0,\pi/2\l] $, and $ \f_{t}(s) \leqslant 1/(1-\etl)$ for all $ s\in [\pi/2\l,\pi-\pi/2\l] $. Moreover, \eqref{d16} helps us write
\begin{align*} 
\int_{ \pi-\pi/2\l }^{ \pi  } \f_{t}(s) \, ds  
& = \int_{ 0 }^{ \pi/2\l  } \dfrac{ \sqrt{1-\ul^2(s)} \, ds}{ 1 \pm \ul(s) \cos(t)  } 
\leqslant  \int_{ 0 }^{ \pi/2\l  } \dfrac{ \sqrt{1-\ul^2(s)} \, ds}{ 1 - \ul(s) \cos(t)  }
\leqslant \int_{ 0 }^{ \pi/2\l } \frac{\sin(\l s) \, ds }{  1 - \cos(s) \cos(t) } \notag \\
& \leqslant \frac{\l \pi}{2} \int_{ 0 }^{ \pi/2\l } \frac{\sin(s) \, ds }{  1 - \cos(s) \cos(t) }
= \frac{\l \pi}{2\cos(t)} \log \left(\frac{1-\cos(\pi/2\l)\cos(t)}{1-\cos(t)}\right) < \infty ,
\end{align*}
therefore, $ \f_{t} $ is integrable over $ [0,\pi] $.
We also need to show that $ \f_{t} $ is a function of bounded variation over $ [0,\pi] $. 
We observe that
\begin{align*} 
\frac{d \f_{t}(s)}{d s} 
& = \left( \dfrac{ \ul(s) \cos(s)-\cos(\l s) }{\sin(s) \sqrt{1-\ul^2(s)}} \right)
\! \left( \dfrac{ \ul(s) }{1+ \ul(s) \cos(t)} + \dfrac{ \left(1-\ul^2(s) \right) \cos(t) }{\left(1+ \ul(s) \cos(t)\right)^2} \right) \notag \\
& = \h_\l(s)
\left( \dfrac{ \ul(s) }{1+ \ul(s) \cos(t)} + \dfrac{ \left(1-\ul^2(s) \right) \cos(t) }{\left(1+ \ul(s) \cos(t)\right)^2} \right) 
\ncomma (s,t) \in (0,\pi) ^2.
\end{align*}
It is also easy to check that for $ s \in (0, \pi/2] $,
\begin{align*} 
\h_\l(\pi-s) 
& = \begin{cases}
\h_\l(s) ,  & \text{if $\l$ is even,} \\
-\h_\l(s) , & \text{if $\l$ is odd.}
\end{cases}
\end{align*}
If $ t=\pi/2 $, \eqref{d15} and \eqref{d47-3} imply that
\begin{align*} 
\int_{ 0 }^{ \pi  } \abs{\frac{d \f_{\pi/2}(s)}{d s}} ds  
& = \int_{ 0 }^{ \pi  } \abs{\h_\l(s)\, \ul(s)} ds
\leqslant \int_{ 0 }^{ \pi  } \abs{\h_\l(s)} ds
= 2 \int_{ 0 }^{ \pi/2\l  } \abs{\h_\l(s)} ds
+  \int_{ \pi/2\l }^{ \pi-\pi/2\l  } \abs{\h_\l(s)} ds \notag \\
& \leqslant \frac{\left(\l^2-2\right)\pi }{2\l} + \int_{ \pi/2\l }^{ \pi-\pi/2\l  }  \dfrac{2\, ds}{\sin(s) \sqrt{1-\etl^2}}
\leqslant \frac{\left(\l^2-2\right)\pi }{2\l} + \frac{2\left(\l-1\right)\pi }{\sqrt{1-\etl}}<\infty.
\end{align*}
For $ t \in (0,\pi/2) $, we observe that
\begin{align} \label{e5-1-1}
& \abs{\frac{d \f_{t}(s)}{d s} }
\leqslant 2 \h_\l(s) \leqslant \l^2-2 \ncomma s \in [0,\pi/2\l],
\end{align}
and 
\begin{align} \label{e5-1-2}
\abs{\frac{d \f_{t}(s)}{d s}} 
& \leqslant \abs{\h_\l(s)}
\left(\dfrac{ 2+\abs{\ul(s)} }{ 1-\abs{\ul(s)}}\right) 
\leqslant \frac{6}{\sin(s) \sqrt{1-\etl^2} \left(1-\etl\right)} 
\leqslant \frac{6\l}{\left(1-\etl\right)^{3/2}} \ncomma s \in [\pi/2\l,\pi-\pi/2\l].
\end{align}
Therefore, it remains to show that $ \f_{t}$ is a function of bounded variation 
over $ [\pi-\pi/2\l,\pi] .$ By \eqref{d16} and \eqref{d47-3}, we have
\begin{align} \label{e2}
\int_{ \pi-\pi/2\l }^{ \pi  } \abs{\frac{d \f_{t}(s)}{d s}} ds  
& = \int_{\pi-\pi/2\l }^{ \pi  } \abs{\h_\l(s)}
\abs{\dfrac{ \ul(s) }{1+ \ul(s) \cos(t)} + \dfrac{ \left(1-\ul^2(s) \right) \cos(t) }{\left(1+ \ul(s) \cos(t)\right)^2}} ds \notag \\
& = \int_{0}^{ \pi/2\l } \abs{\pm\h_\l(s)}
\abs{\dfrac{ \pm\ul(s) }{1\pm \ul(s) \cos(t)} + \dfrac{ \left(1-\ul^2(s) \right) \cos(t) }{\left(1\pm \ul(s) \cos(t)\right)^2}} ds \notag \\
& \leqslant \int_{0}^{ \pi/2\l } \h_\l(s)
\left(\dfrac{ \ul(s) }{1- \ul(s) \cos(t)} + \dfrac{ \left(1-\ul^2(s) \right) \cos(t) }{\left(1- \ul(s) \cos(t)\right)^2} \right) ds \notag \\
& \leqslant \frac{\l^2-2}{2} \int_{0}^{ \pi/2\l } 
\left(\dfrac{ 1 }{1- \cos(s) \cos(t)} + \dfrac{ \sin^2(\l s) \cos(t) }{\left(1- \cos(s) \cos(t)\right)^2} \right) ds \notag \\
& \leqslant \frac{\l^2-2}{2} \int_{0}^{ \pi/2\l } 
\left(\dfrac{ 1 }{1- \cos(s) \cos(t)} + \dfrac{\l^2 \pi^2 \sin^2( s) \cos(t) }{4 \left(1- \cos(s) \cos(t)\right)^2} \right) ds \notag \\
& \leqslant \frac{\l^2-2}{2} \int_{0}^{ \pi } 
\left(\dfrac{ 1 }{1- \cos(s) \cos(t)} + \dfrac{\l^2 \pi^2 \sin^2( s) \cos(t) }{4 \left(1- \cos(s) \cos(t)\right)^2} \right) ds \notag \\
& = \frac{\l^2-2}{2} \left(\frac{\pi}{\sin(t)} + \frac{\l^2 \pi^2}{4} \int_{ 0 }^{ \pi  } \frac{ \cos(s) \, ds }{ 1- \cos( s ) \cos( t ) }\right) \com
\end{align}
where the last equality is obtained by \eqref{d18} and integration by parts.
Note that with help of \cite[3.613(1) on p. 366]{GR} and setting $ n=1, $ we have
\begin{align*} 
\int_{ 0 }^{ \pi  } \frac{ \cos(s) \, ds }{  1- \cos( s ) \cos( t )}
& = \frac{\pi(1-\sin(t))}{\sin(t)\cos(t)} ,
\end{align*}
which helps us rewrite \eqref{e2} as
\begin{align*} 
\int_{ \pi-\pi/2\l }^{ \pi  } \abs{\frac{d \f_{t}(s)}{d s}} ds  
& \leqslant \frac{\l^2-2}{2} \left(\frac{\pi}{\sin(t)} + \frac{\l^2 \pi^3(1-\sin(t))}{4\sin(t)\cos(t)} \right) <\infty \com
\end{align*}
which implies that the $ \f_{t} $, $ t \in (0,\pi/2] $, are functions of bounded variation over $ [0,\pi] $. Now, given the above facts about $ \g_{t} $, we can follow the same method as in Section 2 of \cite{C} and obtain 
\begin{align*} 
I(t) - R_{n'}(\g_{t}) = \frac{\pi}{n'} \int_{ 0 }^{ \pi  } \nu_{n'}(s) \, d\g_{t}(s)  
& = \frac{2\pi}{n} \int_{ 0 }^{ \pi  } \nu_{n'}(s) \, d\g_{t}(s)  \com
\end{align*}
where $ \nu_{n'}(s) :=\pi \, \rho_{n'}(s)-n's \ncomma s \in [0,\pi] ,$ is called the modified sawtooth function, and $ \rho_{n'}:=\sum_{ k = 1 }^{n'} \mathds{1}_{G_k} $ with $ \mathds{1}_{G_k} $ denoting the characteristic function on $ G_k=[(2k-1)\pi/n,\pi] $.
In particular, we have
\begin{align} \label{e10}
\abs{\int_{ 0 }^{ \pi  } \left(I(t) - R_{n'}(\g_{t})\right) dt} 
& \leqslant \frac{2\sqrt{3}\pi}{n}  \int_{ 0 }^{ \pi  } \int_{ 0 }^{ \pi  } \abs{\nu_{n'}(s)}  \abs{\frac{d\f_{t}(s)}{d s}} ds \, dt .
\end{align}
One can easily check that $ \abs{\nu_{n'}(s)} \leqslant \pi/2 \ncomma s \in [0,\pi] ,$
hence, it follows from \eqref{e5-1-1} and \eqref{e5-1-2} that
\begin{align} \label{e5-1-3}
\int_{0}^{\pi/2\l} \abs{\nu_{n'}(s)} \int_{ 0 }^{ \pi/2  } \abs{\frac{d \f_{t}(s)}{d s}} dt \, ds  
& \leqslant \frac{(\l^2-2)\pi^3}{8\l},
\end{align}
and
\begin{align} \label{e5-1-4}
\int_{\pi/2\l}^{\pi-\pi/2\l} \abs{\nu_{n'}(s)} \int_{ 0 }^{ \pi/2  } \abs{\frac{d \f_{t}(s)}{d s}} dt \, ds  
& \leqslant \frac{3(\l-1)\pi^3}{2\left(1-\etl\right)^{3/2}}.
\end{align}
Note that $ \abs{\nu_{n'}(\pi-s)} = \abs{\nu_{n'}(s)} , \ s \in [0,\pi] .$ Thus, the same reasoning as in \eqref{e2} gives
\begin{align*} 
& \int_{\pi-\pi/2\l}^{\pi} \abs{\nu_{n'}(s)} \int_{ 0 }^{ \pi/2  } \abs{\frac{d \f_{t}(s)}{d s}} dt \, ds  \notag \\
& \leqslant \frac{\l^2-2}{2} \int_{0}^{ \pi/2\l } \abs{\nu_{n'}(s)}
\int_{0}^{ \pi/2 } \left(\dfrac{ 1 }{1- \cos(s) \cos(t)} + \dfrac{\l^2 \pi^2 \sin^2( s) }{4 \left(1- \cos(s) \cos(t)\right)^2} \right) dt \, ds \notag \\
& \leqslant \frac{\l^2-2}{2} \int_{0}^{ \pi/2\l } \abs{\nu_{n'}(s)}
\int_{0}^{ \pi } \left(\dfrac{ 1 }{1- \cos(s) \cos(t)} + \dfrac{\l^2 \pi^2 \sin^2( s) }{4 \left(1- \cos(s) \cos(t)\right)^2} \right) dt \, ds \notag \\
& \leqslant \frac{\l^2-2}{2} \int_{0}^{ \pi/2\l } \abs{\nu_{n'}(s)}
\left(\frac{\pi}{\sin(s)} + \dfrac{\l^2 \pi^2 \sin^2( s)}{4 }
\int_{0}^{ \pi }\dfrac{ dt }{\left(1- \cos(s) \cos(t)\right)^2} \right) ds .
\end{align*}
We implement \cite[2.554(3) on p. 148]{GR} with $ n=2 \ncomma a=1 $ and $ b=-\cos(s) $ and have 
\begin{align*} 
& \int_{ 0 }^{ \pi  } \frac{  dt }{ \left( 1-\cos( s )\cos( t ) \right)^2 }  = \frac{1}{\sin^2(s)} \int_{ 0 }^{ \pi  } \frac{  dt }{ 1-\cos( s )\cos( t ) }
= \frac{\pi}{\sin^3(s)} ,
\end{align*}
where the latter equality is again achieved by \cite[3.613(1) on p. 366]{GR}.
Combining the last two relations, we observe that
\begin{align*} 
\int_{\pi-\pi/2\l}^{\pi} \abs{\nu_{n'}(s)} \int_{ 0 }^{ \pi/2  } \abs{\frac{d \f_{t}(s)}{d s}} dt \, ds
& \leqslant  \frac{(\l^2-2)(4\pi+\l^2\pi^3)}{8} \int_{ 0 }^{ \pi/2\l } \frac{\abs{\nu_{n'}(s)}}{\sin(s)} \, ds \notag \\
& \leqslant  \frac{(\l^2-2)(4\pi^2+\l^2\pi^4)}{16}
\left( \int_{ 0 }^{ \pi/n } \frac{\abs{\nu_{n'}(s)}}{s} \, ds 
+ \int_{ \pi/n }^{ \pi/2\l } \frac{\abs{\nu_{n'}(s)}}{s} \, ds\right),
\end{align*}
with the last inequality obtained by $ \sin(s) \geqslant 2s/\pi $.
Note that $ \abs{\nu_{n'}(s)} \leqslant \pi/2 \ncomma s \in [\pi/n,\pi/2\l] ,$ and $ \abs{\nu_{n'}(s)} = n's= ns/2 $ on $ [0,\pi/n] $, hence 
\begin{align} \label{e9}
& \int_{\pi-\pi/2\l}^{\pi} \abs{\nu_{n'}(s)} \int_{ 0 }^{ \pi/2  } \abs{\frac{d \f_{t}(s)}{d s}} dt \, ds
\leqslant  \frac{(\l^2-2)(4\pi^3+\l^2\pi^5)\left( \log n  -\log (2\l)+1\right)}{32}
.
\end{align}
It follows from \eqref{e5-1-3}-\eqref{e9} that there exist $ c_\l>0 $ and $ N \in \N $ such that, for all $ n \geqslant N $, 
\begin{align*} 
& \int_{ 0 }^{ \pi } \abs{\nu_{n'}(s)} \int_{ 0 }^{ \pi/2  } \abs{\frac{d \f_{t}(s)}{d s}} dt \, ds 
\leqslant c_\l \log n.
\end{align*}
Comparably, we can show that
\begin{align*} 
& \int_{ 0 }^{ \pi } \abs{\nu_{n'}(s)} \int_{ \pi/2 }^{ \pi  } \abs{\frac{d \f_{t}(s)}{d s}} dt \, ds 
\leqslant c_\l \log n.
\end{align*}
Therefore, \eqref{e10} and these two estimates imply that, for all $ n \geqslant N $, we have
\begin{align*} 
\abs{\int_{ 0 }^{ \pi  } \left(I(t) - R_{n'}(\g_{t})\right) dt} 
& \leqslant \frac{2\sqrt{3}\pi}{n}  \int_{ 0 }^{ \pi  } \int_{ 0 }^{ \pi  } \abs{\nu_{n'}(s)} \abs{\frac{d \f_{t}(s)}{d s}} ds \, dt \notag \\
& = \frac{2\sqrt{3}\pi}{n} \int_{ 0 }^{ \pi } \abs{\nu_{n'}(s)} \int_{ 0 }^{ \pi  } \abs{\frac{d \f_{t}(s)}{d s}} dt \, ds 
\leqslant  \frac{4\sqrt{3}\pi c_\l \log n}{n} \com
\end{align*}
where the interchange of integration order is justified by the Fubini-Tunelli Theorem.
Thus, \eqref{e12} holds, so does \eqref{e1}. 
Now, combining \eqref{e1} with the fact that $ \g(s,t) $ is symmetric about the line $ \langle\pi/2,\pi/2,r\rangle $, $ r \in \R $, gives us
\begin{align*} 
\int_{ 0 }^{ \pi  } \frac{\pi}{n} \sum_{k=1}^{n/2} \sqrt{1+  \dfrac{3 \left(1-\ul^2(\zk/n)\right) }{  \left(1+ \ul(\zk/n) \cos(t) \right) ^2 }} \, dt  
& = \int_{ 0 }^{ \pi  } \int_{ 0 }^{ \pi/2 } \sqrt{1+  \dfrac{3 \left(1-\ul^2(s)\right) }{  \left(1+ \ul(s) \cos(t) \right) ^2 }} \, ds 
\, dt 
+  \O  \! \left( \frac{\log n}{n} \right) .
\end{align*}
Finally, this very last estimate and \eqref{e14} lead us to 
\begin{align*} 
I_{\l}(n)
& =  \frac{1}{\pi^2}  \int_{ 0 }^{ \pi  } \frac{\pi}{n} \sum_{k=1}^{n/2} \sqrt{1+  \dfrac{3 \left(1-\ul^2(\zk/n)\right) }{  \left(1+ \ul(\zk/n) \cos(t) \right) ^2 }} \, dt 
+  \O  \! \left(n^{-1/3}\right) \notag \\
& = \frac{1}{\pi^2}  \int_{ 0 }^{ \pi  } \int_{ 0 }^{ \pi/2 } \sqrt{1+  \dfrac{3 \left(1-\ul^2(s)\right) }{  \left(1+ \ul(s) \cos(t) \right) ^2 }} \, ds \, dt
+  \O  \! \left( \frac{\log n}{n} \right)  +  \O  \! \left(n^{-1/3}\right)
\notag \\
& = \frac{1}{\pi^2}  \int_{ 0 }^{ \pi  } \int_{ 0 }^{ \pi/2 } \sqrt{1+  \dfrac{3 \left(1-\ul^2(s)\right) }{  \left(1+ \ul(s) \cos(t) \right) ^2 }} \, ds 
\, dt +  \O  \! \left(n^{-1/3}\right) 
= \cK_{\l} +  \O  \! \left(n^{-1/3}\right) \asntoinfty.
\end{align*}
Hence, \eqref{d36} holds as required. At last, \eqref{d36} together with setting $ a=1/3 $ in \eqref{d35} concludes the proof.
\end{proof}

\section*{Acknowledgment}
I am grateful to my advisor Igor Pritsker for leading me in the direction of this project and for all his helpful comments, suggestions and hints which significantly improved the paper. This manuscript is one of the main contributions to the author's Ph.D. dissertation under his supervision. The author would also like to thank the Vaughn Foundation, via Anthony Kable, for financial support.

\bibliographystyle{amsplain}

\end{document}